\documentclass[11pt]{amsart}
\usepackage{graphics,verbatim,color}

\usepackage[all]{xy}

\newtheorem*{theorem*}{Theorem}
\newtheorem{proposition}{Proposition}
\newtheorem*{proposition*}{Proposition}
\newtheorem{lemma}{Lemma}  
 
\newtheorem{definition}{Definition}

\newtheorem*{conjecture*}{Conjecture}

\newtheorem*{problem*}{Problem}

\newtheorem*{example*}{Example}

\def\cal{\mathcal}

\def\R{\mathbb{R}}

\def\g{\gamma}

\def\r{\rho}

\def\cV{{\mathcal V}}

\def\hB{\hat B}
\def\hG{\hat G}
\def\hD{\hat D}
\def\hR{\hat R}

\def\Dstar{{\mathcal D}^*}

\def\smallskip{\par\vspace{1mm}}
\def\medskip{\par\vspace{2mm}}
\def\bigskip{\par\vspace{3mm}}

\def\fr#1#2{\frac{#1}{#2}}
\def\smfr#1#2{\tfrac{#1}{#2}}

\def\m#1{\begin{bmatrix}#1\end{bmatrix}}

\newcount\equationcount
\def\thenumber{0}
\def\eq#1{\global\advance\equationcount by 1
   \def\thenumber{\number\equationcount}
                        {$$#1\eqno(\thenumber)$$}}

\begin{document}

\title[Partially rigid motions]{Partially rigid motions in the $n$-body problem}
\author{Richard Moeckel}
\address{School of Mathematics\\ University of Minnesota\\ Minneapolis MN 55455}

\email{rick@math.umn.edu}



\begin{abstract}
A solution of the $n$-body problem in $\R^d$ is a {\em relative equilibrium} if all of the mutual distance between the bodies are constant.  In other words, the bodies undergo a rigid motion.  Here we investigate the possibility of partially rigid motions, where some but not all of the distances are constant.  In particular, a {\em hinged} solution is one such that exactly one mutual distance varies.  The goal of this paper is to  show that hinged solutions don't exist when $n=3$ or $n=4$. For $n=3$ this means that if 2 of the 3 distances are constant so is the third and for $n=4$, if 5 of the 6 distances are constant, so is the sixth.  These results hold independent of the dimension $d$ of the ambient space.
\end{abstract}

\date{June 21, 2024}
\maketitle

\section{Introduction}
Lagrange  discovered the {\em relative equilibrium} solutions of the three-body problem in $\R^3$, characterized by the  property that  all three mutual distances remain constant.   Up to rotation and scaling there are exactly four possible shapes for a relative equilibrium configuration -- the equilateral triangle configuration and three collinear configurations, one for each choice of which mass is between the other two.  

In this paper, we consider solutions where only two of the mutual distances are constant.
\begin{definition}
A solution of the three-body problem will be called {\em hinged} if exactly two of the three mutual distances remain constant.  More generally, a solution of the $n$-body problem is hinged if all but one of the mutual distances are constant and the last one is not.
\end{definition}
  
The main result here is that hinged solutions  do not exist for $n=3,4$.
\begin{theorem*}\label{th_unhinged}
The three- and four-body problems in $\R^d$ are unhinged.  In other words, for $n=3$, the only solutions such that two of the mutual distances remain constant are the relative equilibrium solutions.  For $n=4$, the only solutions such that five of the six distance are constant are the relative equilibria.
\end{theorem*}
Note that for $n=3$, the theorem is false if two of the masses are zero.  If $m_1=m_2=0$ we can put them on circular orbits around $m_3$ to produce hinged solutions.  In fact this is the standard model of a clockwork solar system with two planets.  The theorem shows that such circular motions are impossible if the masses are all positive.

\begin{figure}[h]
\scalebox{0.6}{\includegraphics{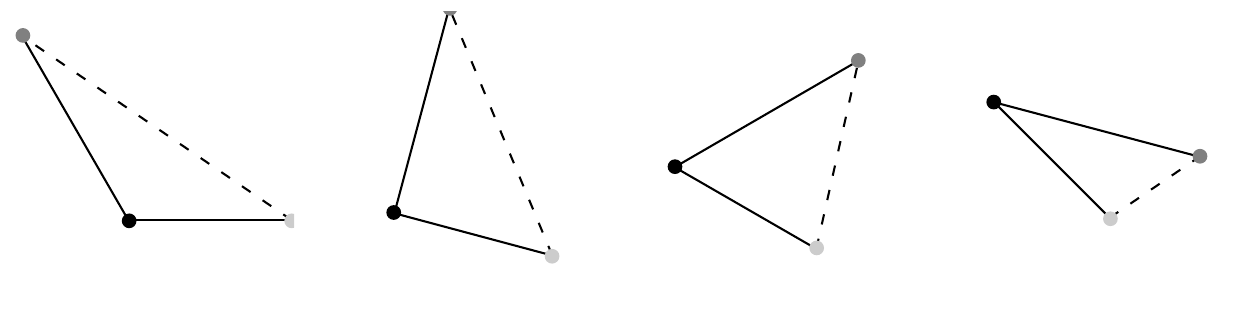}}
\caption{Hinge motion of three bodies in a plane.  Two distances are constant and the third changes.}
\label{fig_hinge3}
\end{figure}

Figures~\ref{fig_hinge3} and \ref{fig_hinge4} indicate some hinged motions of three bodies in a plane and four bodies in space and motivates the terminology.  But these motions can't be solutions of the $n$-body problem in any $\R^d$.  For four bodies in the plane, it's clear that fixing five distance  produces a rigid structure on purely geometrical grounds.  For $n=5$, fixing nine of the 10 distances imposes a lot of rigidity and it would be necessary to go to $\R^4$ to construct a hinge-like motion.    However, it's natural to conjecture that the theorem is still true for $n>4$, that is, that the $n$-body problem is unhinged.

\begin{figure}[h]
\scalebox{0.6}{\includegraphics{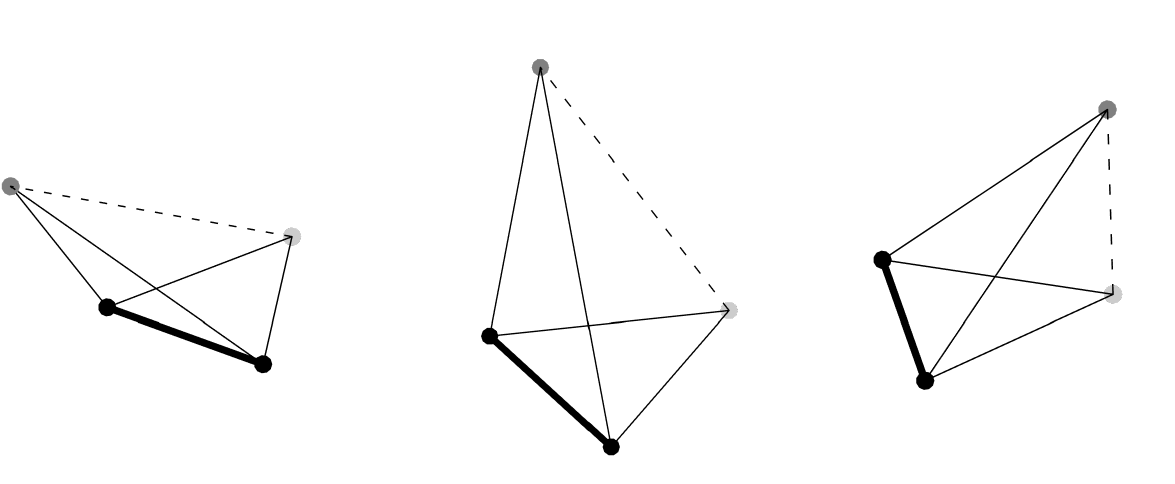}}
\caption{Hinge motion of four bodies in space.  Five distances are constant and the sixth changes.  It can be viewed as two rigid triangles meeting along a hinge axis (bold line) containing two of the bodies.}
\label{fig_hinge4}
\end{figure}

A more interesting conjecture is that for the three-body problem it suffices to fix only one distance to obtain the conclusion
\begin{definition}
A solution of the $n$-body problem will be called {\em  partially rigid} if at least one, but not all, of the mutual distances remains constant.
\end{definition}
\begin{conjecture*}
There are no partially rigid solutions of the three-body problem in $\R^d$.  In other words, the only solutions such that one of the mutual distances remains constant are the relative equilibrium solutions.
\end{conjecture*}
The conjecture is false if even one of the masses is zero.  In fact, for the well-known circular restricted  three-body problem with $m_3=0$, all of the solutions have $r_{12}(t)=1$.   However, it's possible to prove  that the planar circular restricted three-body problem is unhinged.  In other words, the only solutions such that two of the mutual distances remain constant are the relative equilibrium solutions.  But the proof will not be given here.

The motivation for these ideas came from considering a simple model of tidal friction.  Imagine that two of the three bodies are connected by a massless spring with friction. When the length of the spring is changing, energy is dissipated and it is natural to conjecture that the system will converge to some kind of equilibrium state where the length of the spring remains constant.   Is this state necessarily a relative equilibrium or could there still be relative motions among the bodies~?  The conjecture about lack of coupled solutions is just the analogous problem without the spring. 

 Clearly there are many variations on this theme of partially rigid motions.  A general formulation would be as follows.
 \begin{problem*}
 For solutions of the $n$-body problem in $\R^d$, what is the smallest number mutual distances which must be assumed fixed in order to guarantee  that  the motion is a relative equilibrium ?  Call it the {\em celestial rigidity number} $\r(n,d)$. 
 \end{problem*}
 If the dimension $d$ is small there will be some geometrical rigidity giving an upper bound for $\r(n,d)$.  But it may be that we always have $\r(n,d)=1$.

\section{The Lagrange-Albouy-Chenciner equations}
Lagrange was able to eliminate the translational and rotational symmetry of the $3$-body problem by deriving differential equations for the mutual distances \cite{Lag}.  More recently Albouy and Chenciner extended this to the $n$-body problem in $\R^d$ \cite{AlbChen, Alb, Chen, Moe}.  In this section we present a version of their method and derive equations satisfied by hypothetical hinged solutions.

Consider the $n$-body problem in $\R^d$ and let $X_i, V_i\in \R^d$ and $m_i>0$ denote the positions, velocities and masses of the $i$-th body for $i=1,\ldots,n$.   Let $X, V$ be the $d\times n$ matrices whose columns are the vectors $X_i, V_i$.   To eliminate the rotational symmetry,construct the  $n\times n$ Gram matrices 
$$B=X^TX\qquad C= X^TV\qquad D= V^TV.$$
Since the entries are Euclidean inner products, these are invariant under rotations of $\R^d$.  The matrices $B$ and $D$ are symmetric and it is convenient to split the matrix $C$ into symmetric and antisymmetric parts
$$C= G+R\qquad G=\fr12(C+C^T)\quad R = \fr12(C-C^T).$$
We can regard $B,G,D,R$ as representing bilinear forms on $\R^n$.

Newton's law of motion for the $n$-body problem implies the matrix differential equations
$$\dot X= V\qquad \dot V= XA(X)$$
where 
\begin{equation}\label{eq_A}
A(X) = 
\m{A_{11}& \fr{m_1}{r_{12}^3}&\cdots& \fr{m_1}{r_{1n}^3}\\
 \fr{m_2}{r_{12}^3}&A_{22}&\cdots& \fr{m_2}{r_{2n}^3}\\
  \vdots& & &\vdots \\
  \fr{m_n}{r_{1n}^3}& \fr{m_n}{r_{2n}^3}&\cdots&A_{nn}
  }
  \qquad A_{jj} = - \sum_{i\ne j}A_{ij} = -\sum_{i\ne j} \fr{m_i}{r_{ij}^3}
\end{equation}
and $r_{ij} = |X_i-X_j|$  are the mutual distances between the bodies.
Then it's easy to see that the Gram matrices satisfy
\begin{equation}\label{eq_matrixODE}
\begin{aligned}
\dot B&=2G\\
\dot G&=D+\fr12(BA+A^TB)\\
\dot D&= (GA+A^TG)+(A^TR-RA)\\
\dot R&= \fr12(BA-A^TB).
\end{aligned}
\end{equation}

To eliminate the translation symmetry, note that if we translate all of the positions by some vector $u\in\R^d$ and all velocities by $v\in\R^d$ then $X\mapsto X+uL^T$ and $V\mapsto V+vL^T$  where $L\in\R^n$ is the $n\times 1$ column  vector whose entries are all $1$'s.  It follows that the restriction of our bilinear forms to the $(n-1)$-dimensional hyperplane
$$\Dstar = L^\perp = \{u\in\R^n: u_1+\ldots+u_n=0\}$$
will be invariant under translations as well as rotations.  The notation is from \cite{AlbChen} where $\Dstar$ is the dual vector space of the quotient space $\cal D =\R^n/L$.   In this way, specifying a point of the reduced phase space of the $n$-body problem, amounts to giving four bilinear forms on an $(n-1)$-dimensional vector space.
We will show that the differential equations (\ref{eq_matrixODE}) induce a well-defined dynamical system on this space.

Begin by replacing the Gram matrices in (\ref{eq_matrixODE}) by equivalent ones.  We will call two $n\times n$  matrices {\em equivalent} if  they restrict to the same bilinear form on $\Dstar$. 
\begin{lemma}\label{lemma_reps}  
Every symmetric $n\times n$ matrix is equivalent a unique symmetric matrix with all of its diagonal elements equal to $0$.  Every antisymmetric $n\times n$ matrix is equivalent a unique antisymmetric matrix with all the entries in the last row and column equal to $0$.  
\end{lemma}
\begin{proof}
Let $B$ be any symmetric $n\times n$ matrix.  If we subtract $\fr12 b_{ii}L^T$ from the $i$-th row and $\fr12 b_{jj}L$ from the $j$-th column we obtain a symmetric matrix $\hat B$ equivalent to $B$ with diagonal elements ${\hat B}_{ii}= 0$.  Let $e_i$ denote the standard basis vectors of $\R^n$ and let $e_{ij}=e_i-e_j$.  These vectors are in $\Dstar$ and we have
$$e_{ij}^T{\hat B}e_{ij} = {\hat B}_{ii}+{\hat B}_{jj}-2{\hat B}_{ij} = -2{\hat B}_{ij}.$$
This shows that the entries of the matrix are uniquely determine by the values of the bilinear form on $\Dstar$.

Let $R$ be any antisymmetric $n\times n$ matrix and let $\hat R$ be the matrix with entries 
$${\hat R}_{ij} = R_{ij} -R_{in}-R_{nj}= e_{in}^TRe_{jn}.$$
Then $\hat R$ is antisymmetric and has vanishing last row and column.  It's equivalent to $R$ since it agrees with $R$ on the vectors $e_{in}$, $i=1,\ldots,n-1$  which form a basis  for $\Dstar$.
\end{proof}
Note that the dimension of the space of symmetric bilinear forms on an $(n-1)$-dimensional space like $\Dstar$ is $\binom{n}{2}$
the same as the number of independent nonzero entries in the equivalent matrix $\hat B$ of the lemma.  Similarly, the space of antisymmetric forms on $\Dstar$ has dimension $\binom{n-1}{2}$, the same as the number of independent nonzero entries of $\hat R$.

\begin{example*}
If $B=X^TX$ is the Gram matrix formed from the position vectors then 
$$e_{ij}^TBe_{ij} = |Xe_{ij}|^2 =|X_i-X_j|^2 = r_{ij}^2$$
so ${\hat B}_{ij}=-\fr12 r_{ij}^2$.  This makes clear how the equivalence class of $B$ represents the configuration of the bodies up to symmetry.  
\end{example*}
Note that the matrix $A(X)$  in (\ref{eq_A}) depends only on the distances $r_{ij}$. Since $r_{ij}^2=e_{ij}^TBe_{ij} = B_{ii}+B_{jj}-2B_{ij}$, we can view $A(X)$ as a function $A(B)$, at least if $B$ is positive definite on $\Dstar$.  Note for later use that, since $e_{ij}\in\Dstar$, $A(B)$ depends only on the equivalence class of $B$.   Taking this point of view, we can forget about $X,V$ and think of (\ref{eq_matrixODE}) as a dynamical system on the space 
$\cV=\{(B,G,D,R)\}$ of quadruples of $n\times n$ matrices with $B,G,D$ symmetric and $R$ antisymmetric and $B$ positive definite on $\Dstar$.   

The real numbers $b_{ij} = e_{ij}^T B e_{ij}, g_{ij} =e_{ij}^T G e_{ij}, d_{ij} = e_{ij}^T D e_{ij}$, $1\le i<j \le n$ and $\r_{ij} = e_{in}^TR e_{jn}$, $1\le i<j<n$ depend only on the equivalence classes of $B,G,D,R$.   We will call them  the {\em standard coordinates} of the  classes.  Note that the components the matrices $\hB, \hG, \hD, \hR$ of Lemma~\ref{lemma_reps} depend only on these standard coordinates.   Indeed, we have $\hB_{ij} = -\fr12 b_{ij}$ and similarly for the symmetric matrices $\hG, \hD$ while $\hR_{ij} = \r_{ij}$.  Using (\ref{eq_matrixODE}), we will derive differential equations for the standard coordinates.  These will describe the dynamics induced by (\ref{eq_matrixODE}) on equivalence classes and will also be useful for the computations later in the paper.

\begin{proposition}\label{eq_reducedODE}
The standard coordinates satisfy
\begin{equation}\label{eq_coordinateODE}
\begin{aligned}
\dot b_{ij}&=2g_{ij}\\
\dot g_{ij}&=d_{ij}+\fr12e_{ij}^T(\hB A+A^T\hB)e_{ij}\\
\dot d_{ij}&= e_{ij}^T(\hG A+A^T\hG)e_{ij}+e_{ij}^T(A^T\hR-\hR A)e_{ij}\\
\dot \r_{ij}&= \fr12e_{in}^T(\hB A-A^T\hB)e_{jn}.
\end{aligned}
\end{equation}
\end{proposition}
Since $A$ depends only on the distances $r_{ij}$ and $b_{ij}=r_{ij}^2$, the right-hand side depends only on the standard coordinates.
\begin{proof}
Let $B(t),G(t),D(t),R(t)$ be any solution of (\ref{eq_matrixODE}) in $\cV$.
Multiply the first three equations of (\ref{eq_matrixODE}) by $e_{ij}^T$ on the left and $e_{ij}$ on the right and the fourth equation of (\ref{eq_matrixODE}) by
$e_{in}^T$ on the left and $e_{jn}$ on the right

The first equation of (\ref{eq_coordinateODE}) follows immediately.  To get the second equation we need to show that we can replace $B$ by $\hB$ in $e_{ij}^T(BA + A^TB)e_{ij}$.  The key to the proof is the fact that
$L^TA =0$ which is easy to see from the definition (\ref{eq_A}) and expresses the translation invariance of Newton's laws of motion.   It follows  that the image of $A$
in contained in $\Dstar$.  Then since $B$ and $\hB$ agree as bilinear forms on $\Dstar$ and $e_{ij},Ae_{ij}\in\Dstar$ we have
$$\begin{aligned}
e_{ij}^T(BA + A^TB)e_{ij} &= e_{ij}^TB(Ae_{ij}) + (Ae_{ij})^TBe_{ij}\\
&= e_{ij}^T\hB(Ae_{ij}) + (Ae_{ij})^T\hB e_{ij} = e_{ij}^T(\hB A + A^T\hB)e_{ij}.
\end{aligned}$$

The same reasoning shows that we can replace $B,G,R$ by $\hB,\hG,\hR$ in the third and fourth equations.
\end{proof}

Lagrange used his reduced differential equations to find the relative equilibrium solutions of the three-body problem in $\R^3$.  Albouy and Chenciner generalized this idea to the $n$-body problem in $\R^d$.  Call $(X,V)$ a {\em relative equilibrium state} if the standard coordinates  form an equilibrium of the reduced differential equations (\ref{eq_coordinateODE}).   It turns out that this is equivalent to requiring that all of the mutual distances $r_{ij}(t)$ are constant, that is, we have a rigid motion (see \cite{AlbChen},\cite{Chen} or \cite{Moe} for a proof).  In particular, the antisymmetric matrix $\hR(t)$ is constant.  Setting $\dot \r_{ij}=0$, $1\le 1<j<n$ shows that the  the configuration $X$ satisfies
\begin{equation}\label{eq_balanced}
B A=A^T B\qquad \text{as bilnear forms on }\Dstar
\end{equation}
In other words, the bilinear form $BA$ is symmetric when restricted to $\Dstar$ .  A configuration satisfying (\ref{eq_balanced}) is called a {\em balanced configuration}.  Thus the configuration of a relative equilibrium state is balanced and the converse is also true -- every balanced configuration gives rise to a relative equilibrium motion in some $\R^d$.  The balanced configurations include the planar  central configurations but there are many balanced configurations which are not central.  For example,  when two masses are equal, any isosceles triangle is balanced.  The corresponding relative equilibrium motions take place in $\R^4$.

In the rest of this paper we want to investigate partially rigid motions, where only some of the mutual distance are assumed to be constant. In particular, for hinged solutions, all but one distance is assumed fixed.  For $n=3,4$ we will  show that  the remaining distance is also constant and so the motion is actually a relative equilibrium after all.  To study partially rigid motions we will need to derive explicit differential equations for the individual distances $r_{ij}(t)$.

Since $b_{ij}=r_{ij}^2$, equations  (\ref{eq_coordinateODE}) give
$\dot b_{ij} = 2r_{ij}\dot r_{ij} = 2g_{ij}$ or
\begin{equation}\label{eq_rijdot}
\dot r_{ij} = \frac{g_{ij}}{r_{ij}}\qquad 1\le i<j\le n.
\end{equation}
It follows that if one of the mutual distances $r_{ij}(t)$ is constant along a solution, then $g_{ij}(t)=0$  for all $t$.  Then the derivatives of $g_{ij}$ must  also vanish.

We have $\dot g_{ij} = d_{ij}+\fr12e_{ij}^T(\hB A+A^T \hB)e_{ij}$ and
\begin{equation}\label{eq_Gddot}
\ddot g_{ij}= 2e_{ij}^T(\hG A+A^T\hG)e_{ij} + e_{ij}^T(A^T\hR-\hR A)+\fr12e_{ij}^T(\hB\dot A+\dot A^T\hB)e_{ij}.
\end{equation}
Since $A(X)$ depends only on the mutual distances we have
$$\dot A = \sum_{i<j}A_{r_{ij}}\dot r_{ij} =  \sum_{i<j}A_{r_{ij}}g_{ij}/r_{ij}$$
where $A_{r_{ij}}$ are the partial derivative matrices of $A$.  We will study equation in detail  for the three- and four-body problems below.  For now we  just observe that for a given configuration $X$, equation (\ref{eq_Gddot}) is a linear equation for the $g_{ij}, \r_{ij}$.  In particular, it does not involve the $d_{ij}$.

\section{The Three-Body Problem is Unhinged}
For  $n=3$ the matrices $\hB,\hG,\hD,\hR$ are $3\times 3$ and their entries are related to the standard coordinates by
$$\hB=-\fr12\m{0&r_{12}^2&r_{13}^2\\r_{12}^2&0&r_{23}^2\\r_{13}^2&r_{23}^2&0}, \hG=-\fr12\m{0&g_{12}&g_{13}\\g_{12}&0&g_{23}\\g_{13}&g_{23}&0}, \hR=\m{0&\r_{12}&0\\-\r_{12}&0&0\\0&0&0}.$$

Suppose we have a hinged solution with $r_{12}(t)$ and $r_{13}(t)$ both constant.  Then $g_{12}(t)=g_{13}(t)=0$.  We have
$$
A = 
\m{- \fr{m_2}{r_{12}^3}- \fr{m_3}{r_{13}^3}& \fr{m_1}{r_{12}^3}& \fr{m_1}{r_{13}^3}\\
 \fr{m_2}{r_{12}^3}&-\fr{m_1}{r_{12}^3}-\fr{m_3}{r_{23}^3}& \fr{m_2}{r_{23}^3}\\
  \fr{m_3}{r_{13}^3}& \fr{m_3}{r_{23}^3}&- \fr{m_1}{r_{13}^3}- \fr{m_2}{r_{23}^3}
  }\qquad 
  \dot A = \frac{3g_{23}}{r_{23}^5}\m{0&0&0\\0&m_3&-m_2\\ 0&-m_3&m_2}
$$

Substituting all this into (\ref{eq_Gddot}) and bracketing by $e_{12}^T, e_{12}$ and  by $e_{13}^T, e_{13}$ gives
\begin{equation}\label{eq_le3bp}
\begin{aligned}
\ddot g_{12}&=\left(2m_3(r_{13}^{-3}-x^{-3})\right)\r_{12}+\fr{m_3}{2}\left(4r_{13}^{-3}+3(r_{12}^2-r_{13}^2)x^{-5}-x^{-3}\right)g_{23}\\
\ddot g_{13}&=\left(2m_2(x^{-3}-r_{12}^{-3})\right)\r_{12}+\fr{m_2}{2}\left(4r_{12}^{-3}+3(r_{13}^2-r_{12}^2)x^{-5}-x^{-3}\right)g_{23}.
\end{aligned}
\end{equation}
Here $r_{12}, r_{13}$ are fixed positive real numbers and $x=r_{23}>0$ is the remaining mutual distance.  We  want  to show that $x(t)$  is actually constant  too or, equivalently, that $g_23(t)=x(t)\dot x(t)=0$ for all $t$.  Suppose not.  Then $g_{23}(t_0)\ne 0$ for some $t_0$ and this inequality holds on some interval around $t_0$ $x(t)$ is not constant on this interval.  Then $(\r_{12}(t),g_{23}(t))$ is a nontrivial solution of the linear system (\ref{eq_le3bp}) and its determinant  must vanish. The determinant is

For any choice of the constants $r_{12}, r_{13}$ the equation $p=0$ gives a nontrivial polynomial constraint on $x$.   This forces $x(t)$ to be constant, a contradiction.
Thus $x(t)$ must be constant and we have a relative equilibrium.  This completes the proof for $n=3$.

\section{The Four-Body Problem is Unhinged}
For  $n=4$ the matrices $B,G,R,D$ are $4\times 4$ and we may assume they are of the form given by Lemma~\ref{lemma_reps}. 
Each of the matrices $B,G,D$ has 6 independent entries and, in particular, there are 6 mutual distances 
$$r_{12},r_{13},r_{14},r_{23},r_{24},r_{34}.$$
The antisymmetric matrix $R$ has  3 independent entries  $\r_{12},\r_{13},\r_{23}$.

If there is a hinged solution with all of the distances except  $r_{34}(t)$ constant then $g_{12},g_{13},g_{14},g_{23},g_{24}$ must all vanish together with all of their derivatives.  As above, we will just need to  use the vanishing of the second derivatives.  Due to the complexity of the equations,  it turns out be convenient to introduce the reciprocals of the distances  and some normalizations.  Let
$x = r_{34}^{-1}$ and $k_{ij}=r_{ij}^{-1}$.  Furthermore, we may use scaling symmetry to assume $r_{12}=k_{12}=1$ and $m_4=1$.

We have five linear equations $\ddot g_{ij} =  0$, $(i,j)\ne(3,4)$, in four variables $\r_{12},\r_{13},\r_{23},g_{34}$ with  coefficients  depending on the masses and the reciprocal distances  $k_{ij}$ and  $x$.  The $5\times 4$ matrix of this linear system is
$$A=
\begin{bmatrix}
a_{11} & {m_3} \left({k_{23}}^3-{k_{13}}^3\right) & {m_3} \left({k_{13}}^3-{k_{23}}^3\right) & 0 \\
 \left({k_{23}}^3-1\right) {m_2} & a_{22} & \left({k_{23}}^3-1\right) {m_2} &a_{24} \\
 \left({k_{24}}^3-1\right) {m_2} & {m_3} \left(x^3-{k_{13}}^3\right) & 0 & a_{34} \\
 \left(1-{k_{13}}^3\right) {m_1} & \left({k_{13}}^3-1\right) {m_1} & a_{33} & a_{44} \\
 \left(1-{k_{14}}^3\right) {m_1} & 0 & {m_3} \left(x^3-{k_{23}}^3\right) & a_{54} \\
\end{bmatrix}
$$
where
$$
\begin{aligned}
a_{11}&={m_3} \left({k_{13}}^3-{k_{23}}^3\right)+{k_{14}}^3-{k_{24}}^3\\
a_{22}&=m_1\left(1-{k_{23}}^3\right) +{k_{14}}^3-x^3 \\
a_{33}&=m_1\left(1-{k_{13}}^3\right)+{k_{24}}^3-x^3\\
a_{24}&=\smfr{3x^5}{4}  \left({k_{13}}^{-2}-{k_{14}}^{-2}\right)+{k_{14}}^3-\smfr{x^3}{4} \\
a_{34}&={m_3} \left(\smfr{3x^5}{4}  \left({k_{14}}^{-2}-{k_{13}}^{-2}\right)+{k_{13}}^3-\smfr{x^3}{4}\right)\\
a_{44}&=\smfr{3}{4} x^5\left({k_{23}}^{-2}-{k_{24}}^{-2}\right)+{k_{24}}^3-\smfr{x^3}{4}\\
a_{54}&={m_3} \left(\smfr{3x^5}{4} \left({k_{24}}^{-2}-{k_{23}}^{-2}\right)+{k_{23}}^3-\smfr{x^3}{4}\right)
\end{aligned}
$$

If there is a hinged  solution,  this  system  must   have nontrivial  solutions.  In fact, it must have solutions with $g_{34}\ne 0$ or even, since it's linear, with $g_{34}=1$.  
The existence of such solutions imposes constraints on the coefficients, that is, constraints on the $k_{ij}$ and $x$.  The goal is to show  that for all choices of the positive constants $k_{ij}, m_i$, we end  up with only finitely many choices for  $x$.

A straightforward appeal to $4\times 4$ determinants of $A$ does not work.  There are five such determinants obtained by omitting one row, each  a polynomial in $x$  with coefficients depending on the constants $k_{ij}, m_i$.  Unfortunately,  there exist  values of these constant making all of them identically zero.  This is true, for example, if all five $k_{ij}=1$.  In this case, however, the system of equations becomes extremely simple.  For example the second and third equation with $g_{34}=1$ two give
$$\rho_{13}(x^3-1) = 1-\fr{x^3}{4} \qquad \rho_{13}(x^3-1) = \fr{x^3}{4}-1 $$
This is solvable only if $x^3=4$ so we do indeed get a constraint on $x$ and constant $x$ contradicts $g_{34}\ne 0$.

We will have to consider various cases depending on the values of the constants.  Here are the steps of the argument.
\begin{enumerate}
\item If $x=k_{13}$ or $x=k_{23}$ we have our constraint on $x$.  Otherwise we can solve two of the five equations for  $\rho_{13}$ and $\rho_{23}$.
\item After setting $g_{34}=1$, this leaves three linear equations of the form $a_i\rho_{12}=b_i$, $i=1,2,3$.  Then three $2\times 2$ determinants $f_i = a_jb_k-a_kb_j$ must vanish.
\item Each equation $f_i=0$ imposes a polynomial constraint on $x$ unless all of the coefficients of all three polynomials vanish.  These coefficients form a system $P_i=0$, $i=1,\ldots, 18$, of polynomials in $k_{ij}, m_i$.  Using Groebner bases we show that these polynomials vanish only when all of the $k_{ij}=1$, the case already dealt with above.
\end{enumerate}

Now we will give some details.  For step 1, use the equations $\ddot g_{13}=\ddot g_{23}==0$  whose coefficients are in rows 2 and 4 of $A$.
If $x\ne k_{13}$ and $x\ne k_{23}$ then since these variables are real, the leading coefficients are nonzero and we can solve uniquely for $\r_{13},\r_{23}$.  

For step 2, we eliminate $\r_{13},\r_{23}$ from the remaining equations $\ddot g_{12}=\ddot g_{14}=\ddot g_{24}=0$ and set $g_{34}=1$ giving three linear equations of the form $a_i \rho_{12} =  b_i$ where $a_i, b_i$ are rational functions of $k_{ij}, m_i$.  Here $i=1$ represents the equation derived from $\ddot g_{12}$ and so on. These are quite complicated but their denominators have only the nonzero factors $k_{ij}, m_i, (k_{13}^3-x^3), (k_{23}^3-x^3)$.  Clearing denominators gives equations where $a_i, b_i$ are polynomials.  The determinants $f_i = a_jb_k-a_kb_j$ must vanish if a solution exists, where $(i,j,k)$ is a cyclic permutation of $(1,2,3)$.   Before proceeding to step 3, we factorize the polynomials $f_i$ and the remove nonzero factors of $k_{ij}, m_i, (k_{13}^3-x^3), (k_{23}^3-x^3)$.  Let the resulting polynomials still be called $f_i$.

Each $f_i$ is a polynomial of of degree 9 in x
$$f_i = p^i_0+p^i_3x^3+p^i_5x^5+p^i_6x^6+p^i_8x^8+p^i_9x^9$$
where the coefficients are polynomials in the constants $k_{ij},m_i$.  The terms of degrees 1,2,4,7 are missing so each has 6 nonzero coefficients.  If the constants are chosen such that any one of these coefficients is nonzero, then $x$ will be constrained to a finite set, as required.  Taking all of the coefficients of all three polynomials gives a system of 18 polynomial equations.

For step 3 we need to show that if all 18 polynomials vanish simultaneously, then all of the $k_{ij}=1$.  We have already seen that in that case we must have $5x^3=2$.  To accomplish this we first compute a Groebner basis of a subset of the polynomials.  With the notation above, the subset is $S=\{p^1_0,p^1_9,p^2_0,p^2_6,p^3_6,p^3_9\}$, which are among the simplest of the 18 coefficients.  We compute a Groebner basis $B$ for the ideal generated by $S$ using the degree reverse lexicographic order with the variables taken in the order
$k_{13},k_{14},k_{23},k_{24},m_1,m_2,m_3$.  If $x$ is not  constrained to a finite set, all of the polynomials in the basis $B$ must vanish.
Among these are the  polynomials
$$m_2(m_1+m_2)(k_{24}^3-1)(k_{14}^3-k_{24}^3)\qquad m_2(m_1+m_2)(k_{24}^3-1)(k_{13}^3-k_{23}^3).$$
So we either have $k_{24}=1$ or else $k_{14}=k_{24}$ and $k_{13}=k_{23}$.

Substituting $k_{24}=1$ in the basis $B$ produces the polynomial $(m_1+m_2)m_3(k_{23}^3-1)$ so we must also have $k_{23}=1$.  Making this substitution produces the polynomial $m_1(k_{14}^3-1)$ giving $k_{14}=1$ and this final substitution forces $k_{13}=1$ as well.  So this case reduces  to the situation with all $k_{ij}=1$.

On the other hand, substitution of $k_{23}=k_{13}$ and $k_{24}=k_{14}$ into $B$ produces $k_{24}^3-k_{14}^3$ so we must have $k_{24}=k_{14}$ too.  Both the these together lead to the equation 
$$(m_1+m_2)^2(1+m_1+m_2+m_3)k_{24}^3(k_{24}^3-1)^2(m_3-k_{24}^3+1)=0.$$
Since we already eliminated the case $k_{14}=1$ we must have $m_3=k_{24}^3-1$.  However, making this additional substitution in $B$ leads to the polynomial $(m_1+m_2)k_{23}^3(k_{24}^3-1)$.     This forces $k_{24}=1$ which we have already eliminated.

To summarize step 3, we have shown that either $x$ is constrained to a finite set by or else the positive constants $k_{ij}, m_i$ must satisfy 18 polynomial equations.  But the simplest 6 of these equations force all of $k_{ij}=1$, a case which led to $x^3=4$.  Thus in all cases, there are only finitely many possibilities for the variable distance $r_{34}=1/x$ which must then be constant like the others.  Hinged solutions don't exist.

\end{document}